\pdfoutput=1

\documentclass{amsart}

\usepackage{amsfonts,amssymb,verbatim,amsmath,amsthm,latexsym,textcomp,amscd}
\usepackage{latexsym,amsfonts,amssymb,epsfig,verbatim}
\usepackage{amsmath,amsthm,amssymb,latexsym,graphics,textcomp,hyperref}
\usepackage{graphicx}
\usepackage{color}
\usepackage{url}
\usepackage{psfrag}
\usepackage{amsmath,comment,amscd}
\usepackage{amssymb}
\usepackage{amsthm}

\usepackage{caption}
\usepackage{color}
\usepackage{enumerate}
\usepackage{float}
\usepackage[a4paper]{geometry}
\usepackage{graphicx}
\usepackage{tikz-cd}

\numberwithin{equation}{section}

\newcommand{\define}{\textbf}

\newcommand{\mc}{\mathcal}

\newcommand{\vp}{\varphi}

\newcommand{\set}[1]{\{#1\}}
\newcommand{\suppressthis}[1]{}

\newcommand{\N}{\mathbb N}

\newtheorem{theorem}{Theorem}[section]
\newtheorem{lemma}[theorem]{Lemma}

\theoremstyle{definition}
\newtheorem{example}[theorem]{Example}

\newtheorem{corollary}[theorem]{Corollary}

\newtheorem{definition}[theorem]{Definition}

\DeclareMathOperator{\bool}{Bool}

\DeclareMathOperator{\chain}{Chain}

\begin{document}
\title{A Ramsey Theorem for Graded Lattices}

\author{Abhishek Khetan}
\address{School
	of Mathematics, Tata Institute of Fundamental Research. 1, Homi Bhabha Road, Mumbai-400005, India}

\email{khetan@math.tifr.res.in}

\author{Amitava Bhattacharya}
\address{School
	of Mathematics, Tata Institute of Fundamental Research. 1, Homi Bhabha Road, Mumbai-400005, India}

\email{amitavabhattacharya@protonmail.com}
\subjclass[2010]{05D10}
\keywords{poset, lattice, coloring, homothety}

\maketitle

\begin{abstract}
	We develop a Van der Waerden type theorem in an axiomatic setting of graded lattices and show that this axiomatic formulation can be applied to various lattices, for instance the set partition and the Boolean lattices.
	We derive the Hales-Jewett theorem as a corollary.
\end{abstract}
\tableofcontents

\section{Introduction}

Van der Waerden Theorem is one of the most important and basic results in Ramsey Theory.
It states that if $k$ and $r$ are positive integers, then there is a natural number $N$ such that whenever $\chi:[N]\to [k]$ is a $k$-coloring of the first $N$ natural numbers, there is contained in $[N]$ an arithmetic progression of length $r$ each of whose elements have the same color.
There are many proofs of this result and it has led to many developments in Ramsey theory.
See \cite{graham_rothschild_book} and \cite{graham_rothschild_paper}.

One can interpret Van der Waerden  theorem in the language of posets as follows.
Let $A_n$ denote the linear poset $\set{1, \ldots, n}$ where the poset relation is the usual comparison relation on the natural numbers.
What Van der Waerden theorem says is that if the linear poset $A_N$ is colored with $k$-colors then there is a monochromatic `homothetic' copy of $A_r$ sitting in $A_N$, provided $N$ is sufficiently large.
It is thus natural to ask if one can prove Van der Waerden like theorems for other posets, for instance, the Boolean and the set partition posets.

The notion of homothety does not make sense for all posets (See \cite{stanley_vol_1} for basic notions about posets).
For this we need the posets to be graded.
Roughly speaking, if $P$ and $Q$ are graded posets then a 
homothety of $P$ into $Q$ is a `translated and scaled' copy of $P$ in $Q$ (See Definition \ref{definition:homothety}).
Both the Boolean and the set partition posets are graded posets.
Moreover, both of them are lattices.
This raises the following natural question.
Suppose a large enough Boolean (partition) lattice is colored with $k$-colors, then can we find a monochromatic homothetic copy of a given small Boolean (partition) lattice in the large Boolean lattice.
Further, we would like to insist that this embedding of the smaller lattice into the large lattice preserves the lattice relations.
Our main theorem (Theorem \ref{theorem:main}) shows that this is indeed the case.

A key feature of these lattices is `self-similarity'.
A large Boolean (partition) lattice contains copies of smaller Boolean (partition) lattices.
Thus one would like to prove a general theorem about sequences of lattices which satisfy some kind of self similarity condition.
To this end, we define a notion of a \emph{homothety system} (See Definition \ref{definition:homothety system}) on a sequence of graded lattices in Section \ref{section:definitions and examples} and prove our main theorem in this abstract setting.
Examples \ref{example:partition lattice sequence is good}, \ref{example:boolean lattice is good}, \ref{exmaple:division lattice is good}, \ref{example:chain lattice is good} show that the division lattice, Boolean lattice, chain lattice, and the set partition lattice satisfy these axioms under the \emph{trivial homothety system} (See Definition \ref{definition:trivial homothety system}).

A substantial generalization of the Van der Waerden type results is the Hales-Jewett theorem proved in \cite{hales_jewett_original}.
We deduce the Hales-Jewett theorem in Section \ref{section:hales jewett theorem} as a corollary to our main theorem.
This is done by considering a non-trivial homothety system on the chain lattices.
The Hales-Jewett theorem is precisely the statement that one gets by applying Theorem \ref{theorem:main} to this system.
We hope that more applications can be found using other lattice, especially the set partition lattice.
The case of the set partition lattice is especially interesting due to a theorem proved in \cite{pudlak_toma} which states that every lattice can be embedded in the partition lattice.

Various authors have developed Ramsey theorems for posets.
See for instance \cite{trotter_cubes}, \cite{trotter_ramsey}, \cite{axenovich_walzer}.
However, to the best of the knowledge of the authors, the homothetic and lattice theoretic aspects have not been considered in literature.

\section{Definitions and Examples}
\label{section:definitions and examples}

We will write $\N$ to denote the set of all the positive integers and $\N_0$ to denote the set of all the non-negative integers.

\begin{definition}
	\label{definition:homothety}
	Let $P$ and $Q$ be graded posets with rank functions $r_P$ and $r_Q$ respectively.
	We say that a map $f:P\to Q$ is a \define{homothety} if $f$ satisfies the following properties
	\begin{enumerate}[a)]
		\item $f(p)\leq f(p')$ if and only if $p\leq p'$.
		\item There is $d\geq 1$, which we call the \define{scale factor} of $f$, such that $r_Q(f(p)) - r_Q(f(p')) = d(r_P(p) - r_P(p'))$ for all $p, p'\in P$.
	\end{enumerate}
\end{definition}

Recall that if $P$ and $Q$ are graded posets with rank functions $r_P$ and $r_Q$ respectively, then $r:P\times Q\to \N_0$ defined as $r(p, q) = r_P(p) + r_Q(q)$ is a rank function on $P\times Q$, and hence the product is also a graded poset.
Thus if $f_1:P\to Q_1$ and $f_2:P\to Q_2$ are homotheties, then so is the map $f:P\to Q_1\times Q_2$ which takes $p\in P$ to $(f_1(p), f_2(p))$.
Also note that the composition of two homotheties is again a homothety.

\begin{definition}
	\label{definition:lattice homothety}
	Suppose $P$ and $Q$ are graded lattices.
	A map $f:P\to Q$ is said to be a \define{lattice homothety} if $f$ is a homothety and satisfies the condition
	\begin{equation}
		f(p\vee p') = f(p) \vee f(p') \quad \text{and} \quad f(p\wedge p') = f(p) \wedge f(p')
	\end{equation}
	for all $p, p'\in P$.
\end{definition}

Recall that if $Q_1$ and $Q_2$ are lattices then so is $Q_1\times Q_2$, where
\begin{equation}
	(q_1, q_2)\vee (q_1', q_2) = (q_1\vee q_1', q_2\vee q_2') \quad \text{and} \quad
	(q_1, q_2)\wedge (q_1', q_2) = (q_1\wedge q_1', q_2\wedge q_2')
\end{equation}
Thus if $P, Q_1$ and $Q_2$ are graded lattices and $f_1:P\to Q_1$ and $f_2:P\to Q_2$ are lattice homotheties, then so is the map $f:P\to Q_1\times Q_2$ taking $p$ to $(f_1(p), f_2(p))$.
A special case is the diagonal map $P\to P\times P$ which takes $p$ to $(p, p)$ is a lattice homothety whenever $P$ is a graded lattice.
Again, the composition of two lattice homotheties is a lattice homothety.

\begin{definition}
	\label{definition:homothety system}
	\label{definition:trivial homothety system}
	Let $A(1), A(2), A(3), \ldots$ be a sequence of graded lattices.
	For each $i$ and $j$ let $H_{ij}$ be a collection of lattice homotheties from $A(i)$ to $A(j)$.
	We say that $\mc H=\set{H_{ij}:\ i, j\geq 1}$ is a \define{homothety system} on $A(1), A(2), A(3), \ldots$ if
	\begin{enumerate}
		\item[H1.] The identity map $A(i)\to A(i)$ is in $H_{ii}$ for all $i$.
		\item[H2.] Whenever $\vp\in H_{ij}$ and $\psi\in H_{jk}$, we have $\psi\circ\vp\in H_{ik}$.
	\end{enumerate}
	We will refer to the elements of $\mc H$ as $\mc H$-\define{restricted lattice homotheties}, or simply as \define{restricted lattice homotheties} when there is no ambiguity about the homothety system under consideration.
	When each $H_{ij}$ consists of all possible lattice homotheties between $A(i)$ and $A(j)$, we call $\mc H$ as the \define{trivial homothety system}.
	Despite the name, the trivial homothety system is of interest.
\end{definition}

\begin{definition}
	Let $A(1), A(2), A(3), \ldots$ be a sequence of graded lattices equipped with a homothety system $\mc H$.
	Let $m, n, N\geq 1$ and $\vp:A(m)\times A(n)\to A(N)$  be a lattice homothety.
	We say that $\vp$ is \define{compatible} with $\mc H$ the following two conditions are satisfies.
	\begin{enumerate}
		\item[C1.] Whenever $A(i)\to A(m)\times A(n)$ is a lattice homothety such that each of the compositions
	\begin{equation}
		A(i) \to A(m)\times A(n) \xrightarrow{\text{proj}_1} A(m) \quad \text{ and }\quad A(i) \to A(m)\times A(n)\xrightarrow{\text{proj}_2} A(n)
	\end{equation}
	are restricted lattice homothety, the composition
	\begin{equation}
		A(i)\to A(m)\times A(n)\xrightarrow{\vp} A(N)
	\end{equation}
	is also a restricted lattice homothety.

		\item[C2.] For any $p\in A(m)$ and $q\in A(n)$, the maps $\vp_p:A(n)\to A(N)$, $\vp_p(y)=\vp(p, y)$ and $\vp_q:A(m)\to A(N)$, $\vp_q(x) = \vp(x, q)$ are both restricted lattice homotheties.
	\end{enumerate}
	Note that when $\mc H$ is the trivial homothety system then every lattice homothety $A(m)\times A(n)\to A(N)$ is compatible with $\mc H$.
\end{definition}

\begin{definition}
	We say that a sequence $A(1), A(2), A(3), \ldots$ of graded lattices equipped with a homothety system $\mc H$ is \define{good} if for all $N_1, N_2\geq 1$, there is $N$ such that there is a lattice homothety from $A(N_1)\times A(N_2)$ into $A(N)$ which is compatible with $\mc H$.
\end{definition}

\begin{definition}
	We say that a sequence $A(1), A(2), A(3), \ldots$ of graded lattices is \define{good} if it is good when it is equipped with the trivial homothety system.
\end{definition}

\begin{example}
	\label{example:partition lattice sequence is good}
	Let $F$ be a finite set.
	We write $\Pi(F)$ to denote the collection of all partitions of $F$ where we write $\pi\leq \tau$ for two partitions $\pi$ and $\tau$ of $A$ if $\tau$ is a refinement of $\pi$.
	Then $r:\Pi(F)\to \N$ defined as $r(\pi) = |\pi|$ is a grading on $\Pi(F)$.
	Also, if $\pi$ and $\tau$ are arbitrary partitions of $F$, then there is a unique coarsest partition which refines both $\pi$ and $\tau$, and there is a unique finest partition which both $\pi$ and $\tau$ refine.
	These are defined to be the join and meet of $\pi$ and $\tau$ respectively, thus endowing $\Pi(F)$ with the structure of a graded lattice and we refer to $\Pi(F)$ as the \define{partition lattice} on $F$.
	We write $\Pi(n)$ to denote $\Pi(\set{1, \ldots, n})$.

	For two disjoint finite sets $F$ and $G$, there is a natural map $\Pi(F)\times \Pi(G)\to \Pi(F\sqcup G)$ which sends $(\pi, \tau)$ to $\pi\cup \tau$.
	It is easy to check that this is a lattice homothety.
	From this observation one sees that the sequence $\Pi(1), \Pi(2), \Pi(3), \ldots,$ is a good sequence.
\end{example}

\begin{example}
	\label{example:boolean lattice is good}
	For any finite set $F$ we write $\bool(F)$ to denote the set of all the subsets of $F$.
	For $S, T\in \bool(F)$, we write $S\leq T$ if $S\subseteq T$, which endows it with a poset structure.
	It is easily checked that the function $r:\bool(F)\to \N_0$ taking $S$ to $|S|$ is a grading on $\bool(S)$.
	Also, the operations of union and intersection play the role of join and meet respectively.
	Thus $\bool(F)$ is a graded lattice which we will refer to as the \define{Boolean lattice} on $F$.
	We write $\bool(n)$ to denote $\bool(\set{1, \ldots, n})$.

	When $F$ and $G$ are two disjoint sets, there is a natural map $\bool(F)\times \bool(G)\to \bool(F\sqcup G)$ taking $(S, T)$ to $S\cup T$.
	It is clear that this is a lattice homothety and thus the sequence $\bool(1), \bool(2), \bool(3), \ldots$ is a good sequence.
\end{example}

\begin{example}
\label{exmaple:division lattice is good}
	For a positive integer $n$ let $D_n$ denote the set of all the non-negative integers which divide $n$.
	For $a, b$ in $D_n$, we write $a\leq b$ if $a|b$.
	The function $r:D_n\to \N_0$ defined as
	\begin{equation}
		a = p_1^{k_1} \cdots p_r^{k_r} \mapsto k_1+ \cdots + k_r
	\end{equation}
	where $p_i$'s are primes, is a rank function on $D_n$.
	Also, the least common multiple and the greatest common divisor serve as join and meet of any two elements.
	Thus $D_n$ is a graded lattice which we refer to as the \define{division lattice} corresponding to $n$.
	Now let $m = p_1^{k_1} \cdots p_r^{k_r}$ and $n=q_1^{\ell_1} \cdots q_s^{\ell_s}$ be two positive integers, where    $p_i$'s and $q_j$'s are distinct primes.
	Let $u_1, \ldots, u_r, v_1, \ldots, v_s$ be pairwise distinct primes and define
	\begin{equation}
		N = u_1^{k_1} \cdots u_r^{k_r} \cdot v_1^{\ell_1} \cdots v_s^{\ell_s}
	\end{equation}
	Consider the map $f:D_m\times D_n\to D_N$ defined as
	\begin{equation}
		f(p_1^{\varepsilon_1} \cdots p_r^{\varepsilon_r},\ q_1^{\delta_1} \cdots q_s^{\delta_s}) = u_1^{\varepsilon_1} \cdots u_r^{\varepsilon_r} \cdot v_1^{\delta_1} \cdots v_s^{\delta_s}
	\end{equation}
	Then $f$ is a lattice homothety, which shows that the sequence $D_1, D_2, D_3, \ldots$ is a good sequence.
\end{example}

\begin{example}
\label{example:chain lattice is good}
	Fix a positive integer $k$.
	Define $\chain_k(F)$ as the set of all maps $F\to \set{0, \ldots, k-1}$.
	We write $f\leq g$ for two elements $f$ and $g$ in $\chain_k(F)$ if $f(x)\leq g(x)$ for all $x\in F$.
	This gives a poset structure to $\chain_k(F)$.
	The map $r:\chain_k(F)\to \N_0$ defined as $r(f) = \sum_{x\in F}f(x)$ is a grading on $\chain_k(F)$.
	Finally, for two elements $f$ and $g$ in $\chain_k(F)$, we have
	\begin{equation}
		f\vee g = \max(f, g)\quad \text{and} \quad f\wedge g = \min(f, g)
	\end{equation}
	serving as meet and join of $f$ and $g$.
	Thus $\chain_k(F)$ is a graded lattice.
	We will refer to $\chain_k(F)$ as the $k$-th \define{chain lattice} over $F$.

	Now let $F$ and $G$ be two disjoint sets.
	Consider the map $\chain_k(F)\times \chain_k(G)\to \chain_k(F\sqcup G)$ which takes $(f, g)$ to the element $h\in \chain_k(F\sqcup G)$ defined as
	\begin{equation}
		h(x)
		=
		\left\{
		\begin{array}{lcl}
			f(x) \text{ if } x\in F\\
			g(x) \text{ if } x\in G
		\end{array}
		\right.
	\end{equation}
	It can be checked that this is a lattice homothety, which shows that the sequence
	$$
	    \chain_k(1), \chain_k(2), \chain_k(3), \ldots
	$$
	is a good sequence for any $k$.
	Note that $\chain_2(n)$ is same as $\bool(n)$.
\end{example}

\section{Main theorem}

Throughout this section let $A(1), A(2), A(3), \ldots$ be a sequence of graded lattices equipped with a homothety system $\mc H$
and assume that this system is good.

\begin{theorem}
	\label{theorem:main}
	Let $n$ and $k$ be positive integers.
	Then there is a positive integer $N$ such that whenever $\chi:A(N)\to [k]$ is a $k$-coloring of $A(N)$, there is a restricted lattice homothety $h:A(n)\to A(N)$ whose image is $\chi$-monochromatic.
\end{theorem}

Before getting into the proof of the above theorem we develop some definitions.
Let $P$ be a graded poset.
Let $S\subseteq P$ and $p, p'\in P$ be arbitrary.
We say that the triple $(S;p, p')$ is \define{admissible} if 
\begin{enumerate}[a)]
	\item $S$ has a minimum element.
	\item Both $p$ and $p'$ lie in $S$ and are distinct.
	\item $p$ is the minimum element of $S$. 
\end{enumerate}

Let $P$ and $Q$ be graded lattices and $\chi:Q\to [k]$ be a $k$-coloring of $Q$.
Let $q, q'$ be two elements in $Q$ such that $q<q'$ and $(S; p, p')$ be an admissible triple in $P$.
A $((S;p\to q, p'\to q'), \chi)$-\define{lattice homothety} from $P$ to $Q$ is a lattice homothety $f:P\to Q$ such that
\begin{enumerate}[a)]
	\item $f(p) = q$ and $f(p') = q'$.
	\item $f(S)\setminus \set{q'}$ is $\chi$-monochromatic.
\end{enumerate}

Let $k, \ell, n$ and $s$ be positive integers such that $n \geq s$ and $\ell, s\geq 2$.
We define $L_n(s, k, \ell)$ as the smallest positive integer (if it exists) $N$ such that whenever $(S; p, p')$ is an admissible triple in $A(n)$ with $|S| = s$, and $\chi:A(N)\to [k]$ is a $k$-coloring, then there is a sequence $q_1< \cdots< q_\ell$ in $A(N)$ such that for all $1\leq i< j\leq \ell$ there is a $((S;p\to q_i, p'\to q_j), \chi)$-restricted lattice homothety from $A(n)$ to $A(N)$.


\begin{lemma}
	\label{lemma:step 1}
	If $L_n(s, k, \ell)$ and $L_n(s, k^{B(N_1)}, 2)$ exist then $L_n(s,k,\ell+1)$ exists, where $N_1 = L_n(s, k, \ell)$ and $B(N_1)$ is the size of $A (N_1)$.
\end{lemma}
\begin{proof}
	Assume $N_1:=L_n(s, k, \ell)$ and $N_2:=L_n(s, k^{B(N_1)}, 2)$ exist.
	By definition of a good sequence there is $N$ such that there is a lattice homothety $\textbf{i}:A(N_1)\times A(N_2)\to A(N)$ compatible with $\mc H$.
	Fix an admissible triple $(S; p, p')$ in $A(n)$ with $|S|=s$ and a $k$-coloring $\chi:A(N)\to [k]$.
	We will show that there is a sequence $q_1< \cdots < q_{\ell+1}$ in $A(N)$ such that for each $1\leq i< j\leq \ell +1$, there is a $((S; p\to q_i, p'\to q_j), \chi)$-restricted lattice homothety from $A(n)$ to $A(N)$.

	The coloring $\chi$ gives an induced $k^{B(N_1)}$ coloring of $A (N_2)$ since the composite $\chi\circ \textbf{i}$ can be thought of as a map $c:A(N_2)\to [k]^{A(N_1)}$.
	For each $y\in A(N_2)$ we thus have a map $c_y:A(N_1)\to [k]$ which takes $x\in A(N_1)$ to $\chi(x, y)$.
	Now using the fact that $L_n(s, k^{B(N_1)}, 2)$ exists, we get $b<b'$ in $A(N_2)$ such that there is a $((S; p\to b, p'\to b'), c)$-restricted lattice homothety $h_2:A(n)\to A(N_2)$.
	Thus $c_y$ is same for all $y\in h_2(S)\setminus \set{b'}$.
	Let $\psi:A(N_1)\to [k]$ denote $c_y$ for any such $y$.
	The existence of $L_n(s, k, \ell)$ now gives a sequence $p_1< \cdots < p_\ell$ in $A(N_1)$ such that for all $1\leq i<j\leq \ell$ there is a $((S; p\to p_i, p'\to p_j), \psi)$-restricted lattice homothety $h_1^{ij}: A(n) \to A(N_1)$.
	Now consider the list
	\begin{equation}
		\textbf{i}(p_1, b) < \textbf{i}(p_2, b) < \cdots < \textbf{i}(p_\ell, b) < \textbf{i}(p_\ell, b')
	\end{equation}
	in $A(N)$.
	We claim that for any two points $y$ and $y'$ with $y< y'$ in this list, there is a $((S; p\to y, p'\to y'), \chi)$-restricted lattice homothety.
	We consider three cases.
	
	\emph{Case 1: Suppose $y=\textbf{i}(p_i, b)$ and $y'=\textbf{i}(p_j, b)$ for some $i<j$.}
	Consider the map $h:A(n) \to A(N)$ defined as $h(x) = \textbf{i}(h_1^{ij}(x), b)$ for all $x\in S$.
	By using (C2) it is clear that this map is a $((S; p\to y, p'\to y', \chi)$-restricted lattice homothety.

	\emph{Case 2: Suppose $y=\textbf{i}(p_i, b)$ and $y'=\textbf{i}(p_\ell, b')$ for some $i<\ell$.}
	Consider the map $h:A(n) \to A(N)$ defined as
	\begin{equation}
		h(x) = \textbf{i}(h_1^{i\ell}(x), h_2(x))
	\end{equation}
	Now if $x, x'\in S$ with $x'\neq p'$, then we have
	\begin{equation}
		\chi\circ \textbf{i}(h_1^{i\ell}(x), h_2(x))
		= c_{h_2(x)}(h_1^{i\ell}(x))
		= \psi(h_1^{i\ell}(x))
		= \psi(h_1^{i\ell}(x'))
		\end{equation}
		giving
	\begin{equation}
		\chi\circ \textbf{i}(h_1^{i\ell}(x), h_2(x))
		= c_{h_2(x')}(h_1^{i\ell}(x'))
		= \chi\circ \textbf{i}(h_1^{i\ell}(x'), h_2(x'))
	\end{equation}
	Thus, using (C1), we see that $h$ is a $((S; p\to y, p'\to y'), \chi)$-restricted lattice homothety and we are done.

	\emph{Case 3: Suppose $y=\textbf{i}(p_\ell, b)$ and $y'=\textbf{i}(p_\ell, b')$.}
	Consider the map $h:A(n)\to A(N)$ defined as $h(x) = \textbf{i}(p_\ell, h_2(x))$ for each $x\in A(n)$.
	Then $h$ is a restricted lattice homothety by (C2).
	We show that $h(S)\setminus{\set{y'}}$ is $\chi$-monochromatic.
	So we need to show that if $x$ and $x'$ are in $S$ with $x, x'\neq p'$, then $h(x)$ and $h(x')$ receive the same color under $\chi$.
	Indeed,
	\begin{equation}
		\chi(h(x)) = \chi\circ \textbf{i}(p_\ell, h_2(x)) = c_{h_2(x)}(p_\ell) = c_{h_2(x')}(p_\ell) = \chi\circ \textbf{i}(p_\ell, h_2(x')) = \chi(h(x'))
	\end{equation}
	Thus $h$ is a $((S; p\to y, p'\to y'), \chi)$-restricted lattice homothety and we are done.
\end{proof}

\begin{lemma}
	\label{lemma:step 2}
	If $L_n(s,k,k+1)$ exists for some $2\leq s<n$ and some $k\geq 1$ then $L_n(s+1,k,2)$ exists.
\end{lemma}
\begin{proof}
	Let $N=L_n(s, k, k+1)$ and $\chi:A(N)\to [k]$ be a $k$-coloring.
	Let $(S'; p, p')$ be an admissible triple in $A(n)$ with $|S'|=s+1$.
	We will show that there are $b<b'$ in $A(N)$ such that there is a $((S'; p\to b, p'\to b'), \chi)$-restricted lattice homothety from $A(n)$ to $A(N)$.

	Let $S = S'\setminus\set{p'}$ and $m'$ be an arbitrary element in $S$ different from $p$.
	Such an $m'$ exists since $|S'|=s+1>2$.
	By definition of $N$, there is a sequence $q_1< \cdots < q_{k+1}$ in $A(N)$ such that for each $1\leq i<j\leq k+1$ there is a $((S; p\to q_i, m'\to q_j), \chi)$-restricted lattice homothety $h^{ij}$ from $A(n)$ to $A(N)$.
	By the pigeonhole principle there are $q_i$ and $q_j$ such that $\chi(q_i) =\chi(q_j)$.
	Thus the image of $S$ under the corresponding restricted lattice homothety $h^{ij}$ is $\chi$-monochromatic in $A(N)$.
	Write $b = q_i$ and $b' = h^{ij}(p')$.
	Then we have $h_{ij}$ is a $((S'; p\to b, p'\to b'), \chi)$-restricted lattice homothety from $A(n)$ to $A(N)$ and we are done.
\end{proof}

Fix $n\geq s\geq 2$ and let us write $L_n(s, *, *)$ to mean that $L_n(s, k, \ell)$ exists for all $k$ and all $\ell\geq 2$.
Similarly, we write $L_n(s, *, \ell)$ and $L_n(s, k, *)$.

\begin{lemma}
	\label{lemma:precursor to main theorem base case}
	$L_n(2, *, *)$ is true.
\end{lemma}
\begin{proof}
	First we show that $L_n(2, *, 2)$ is true.
	Let $k\geq 1$ be arbitrary.
	We need to show that $L_n(2, k, 2)$ exists.
	To show this all we need to show is that there is a restricted lattice homothety from $A(n)$ to some $A(N)$.
	To this end, by definition of a good sequence, we know that there is $N$ such that there is a lattice homothety $\textbf{i}$ from $A(n)\times A(n)$ into $A(N)$ which is compatible with $\mc H$.
	By using (C1), (H1), and (H2), we see that the composition of $\textbf{i}$ with the diagonal embedding $A(n)\to A(n)\times A(n)$ is a restricted lattice homothety $A(n)\to A(N)$.
	Now the truth of $L_n(2, *, *)$ follows by repeated applications of Lemma \ref{lemma:step 1}.
\end{proof}

\begin{lemma}
	\label{lemma:precursor to the main theorem}
	Fix $2\leq s\leq |A(n)|$.
	Then $L_n(s, *, *)$ is true.
\end{lemma}
\begin{proof}
	We prove this by induction.
	Lemma \ref{lemma:precursor to main theorem base case} shows that $L_n(2, *, *)$ is true.
	Suppose $L_n(i, *, *)$ is true for $i=2, 3, \ldots, s<|A(n)|$.
	We want to show that $L_n(s+1, *, \ell)$ is true for all $\ell\geq 2$.
	By Lemma \ref{lemma:step 2}, the truth of $L_n(s, *, *)$ implies that $L_n(s+1, *, 2)$ is true.
	Suppose we have shown that $L_n(s+1, *, j)$ is true for $j=2, \ldots, \ell$ for some $\ell\geq 2$.
	Let $k\geq 1$ be arbitrary.
	The truth of $L_n(s+1, *, \ell)$ and $L_n(s+1, *, 2)$ in particular implies that $N:=L_n(s+1, k, \ell)$ and $L_n(s+1, k^{B(N)}, 2)$ exist.
	Thus by Lemma \ref{lemma:step 1} we know that $L_n(s+1, k, \ell+1)$ exists.
	Since $k$ was arbitrary, this means that $L_n(s+1, *, \ell+1)$ is true, finishing the proof.
\end{proof}

\begin{proof}[Proof of Theorem \ref{theorem:main}]
	Let $N=L_n(n, k, k+1)$, which exists by Lemma \ref{lemma:precursor to the main theorem}.
	Fix an arbitrary $k$-coloring $\chi:A(N)\to [k]$.
	Let $S=A(n)$ and $p$ and $p'$ be the minimum and maximum elements of $A(n)$ respectively.
	Thus $(S; p, p')$ is an admissible triple in $A(n)$.
	By definition of $N$, there is a sequence $q_1< \cdots < q_{k+1}$ in $A(N)$ such that for each $1\leq i<j\leq k+1$ there is a $((S; p\to q_i, p'\to q_j), \chi)$-restricted lattice homothety $h$ from $A(n)$ to $A(N)$.
	By the pigeonhole principle, there are $i<j$ such that $q_i$ and $q_j$ receive the same color under $\chi$.
	It is now clear that $h^{ij}$ is the desired restricted lattice homothety.
\end{proof}

\section{Application: The Hales-Jewett Theorem}
\label{section:hales jewett theorem}

Let $t$ and $n$ be positive integers.
Let $N$ be a positive integer and $\mc S=\set{S_1, \ldots, S_n}$ be a collection of pairwise disjoint subsets of $[N]$ with $|S_1| = \cdots = |S_n|$.
Let $\alpha:[N]\to \set{0,\ldots, t-1}$ be a function that vanishes on $\bigcup \mc S$.
We write $\vp_{\mc S, \alpha}$ to mean the function $\chain_t(n)\to \chain_t(N)$ defined as
\begin{equation}
	\vp_{\mc S, \alpha}(a) = a_11_{S_1} + \cdots + a_n1_{S_n} + \alpha
\end{equation}
for all $a=(a_1, \ldots, a_n)\in \chain_t(n)$.
Note that $\vp_{\mc S, \alpha}$ is a lattice homothety with scale factor $k$, where $k$ is the cardinality of any of the $S_i$'s.
We say that a lattice homothety $h:\chain_t(n)\to \chain_t(N)$ is of \define{type HJ} if there exists a family $\mc S$ of pairwise disjoint subsets of $[N]$ of equal size with $|\mc S| = n$, and a function $\alpha:[N]\to \set{0, \ldots, t-1}$ which vanishes on $\bigcup S$, such that $h=\vp_{\mc S, \alpha}$.

Now suppose $n_1\leq n_2\leq n_3$ are positive integers and $h_{12}:\chain_t(n_1)\to \chain_t(n_2)$ and $h_{23}:\chain_t(n_2)\to \chain_t(n_3)$ are homotheties of type HJ.
We will show that the composite $h:=h_{23}\circ h_{12}$ is also a homothety of type HJ.
Let $h_{12}=\vp_{\mc S, \alpha}$ and $h_{23} = \vp_{\mc T, \beta}$, where $\mc S=\set{S_1, \ldots, S_{n_1}}$ with $|S_1|= \cdots  = |S_{n_1}|$ and $\alpha$ vanishes on $\bigcup \mc S$, and $\mc T=\set{T_1, \ldots, T_{n_2}}$ with $|T_1| = \cdots = |T_{n_2}|$ and $\beta$ vanishes on $\bigcup \mc T$.

Define $\gamma = h(0, \ldots, 0)$ and for each $i\in [n_1]$ define $U_i$ as $\bigcup_{p\in S_i}T_p$.
It is then readily checked that all the $U_i$'s are of the same size and are pairwise disjoint, $\gamma$ vanishes on $\bigcup_i U_i$, and 
\begin{equation}
	h(a_1, \ldots, a_{n_1}) = a_11_{U_1} + \cdots +a_{n_1}1_{U_{n_1}} + \gamma
\end{equation}
What we have shown is the following

\begin{lemma}
	Let $t$ be a positive integer and $H_{ij}$ be the set of all the homotheties of type HJ from $\chain_t(i)$ to $\chain_t(j)$.
	Then the family $\mc H=\set{H_{ij}:\ i, j\geq 1}$ is a homothety system on the sequence
	$$\chain_t(1), \chain_t(2), \chain_t(3), \ldots$$
	We will refer to this homothety system as the \define{HJ homothety system} on the sequence of chain lattices.
\end{lemma}

\begin{lemma}
	\label{lemma:rigidity of lattice homotheties}
	The trivial homothety system and the HJ homothety system coincide on the sequence of Boolean lattices.
\end{lemma}
\begin{proof}
	We prove this by induction.
	The lemma is clearly true for $n=1, 2$.
	Suppose the lemma holds for $1, \ldots, n$ and let $h:\bool(n+1)\to \bool(N)$ be a lattice homothety.
	Let $k$ be the scale factor of $h$.
	Define
	\begin{equation}
		h_0 = h|\set{0, 1}^n\times\set{0}\quad \text{and} \quad h_1 = h|\set{0, 1}^n\times\set{1}
	\end{equation}
	Then both $h_1$ and $h_2$ can be thought of as lattice homotheties from $\bool(n)$ to $\bool(N)$.
	By induction, there are families $\mc S=\set{S_1, \ldots, S_n}$ and $\mc T=\set{T_1, \ldots, T_n}$ of pairwise disjoint subsets of $[N]$, and function $\alpha, \beta:[N]\to \set{0, 1}$ which vanish on $\bigcup\mc S$ and $\bigcup\mc T$ respectively such that
	\begin{equation}
		h(a_1, \ldots, a_n, 0) = h_0(a_1, \ldots, a_n) = a_11_{S_1} + \cdots +a_n1_{S_n} + \alpha
	\end{equation}
	\begin{equation}
		h(a_1, \ldots, a_n, 1) = h_1(a_1, \ldots, a_n) = a_11_{T_1} + \cdots + a_n1_{T_n} + \beta
	\end{equation}
	for all $(a_1, \ldots, a_n)\in \bool(n)$.
	Thus
	\begin{equation}
		\label{equation:lattice homothety lemma equation}
		h(a_1, \ldots, a_n, a_{n+1}) = (1-a_{n+1})h_0 + a_{n+1} h_1
	\end{equation}
	Since $h$ has scale factor $k$, each $S_i$ and each $S_j$ has size $k$.

	We claim that $S_i=T_i$ for all $i$.
	To this end, note that the restriction of $h$ to $\set{0, 1}\times \set{0}\times \cdots \times \set{0}\times \set{0, 1}$ can be thought of a lattice homothety from $\bool(2)$ to $\bool(N)$ having scale factor $k$.
	Thus there exist two disjoint sets $P$ and $Q$ of $[N]$ of size $k$, and function $\gamma:[N]\to \set{0, 1}$ which vanishes on $P\cup Q$, such that
	\begin{equation}
		h(a_1, 0, \ldots, 0, a_{n+1}) = a_11_{P} + a_{n+1} 1_{Q} + \gamma
	\end{equation}
	Using Equation \ref{equation:lattice homothety lemma equation} we have
	\begin{equation}
		a_11_{P} + \gamma = h(a_1,0, \ldots, 0) = a_11_{S_1} + \alpha
	\end{equation}
	and
	\begin{equation}
		a_11_P + 1_Q + \gamma = h(a_1, 0, \ldots, 0, 1) = a_11_{T_1} + \beta
	\end{equation}
	which give that $S_1 = P = T_1$.
	It similarly follows that $S_i=T_i$ for all $i$.
	The above working also shows that $\beta = 1_Q+\gamma = 1_Q + \alpha$.

	Thus we have
	\begin{equation}
		h(a_1, \ldots, a_{n+1}) = a_11_{S_1} + \cdots a_{n}1_{S_n} + a_{n+1}1_Q + \alpha
	\end{equation}
	Thus $S_i\cap Q$ is empty for each $i$ and $\alpha$ vanishes on $Q\cup S_1\cup \cdots \cup S_n$.
	This shows that $h=\vp_{\mc F, \alpha}$ where $\mc F=\set{S_1, \ldots, S_n, Q}$ and we are done.
\end{proof}

\begin{example}
	\emph{We show that the trivial homothety system and HJ homothety system may not be the same for chain lattices.}
	Consider the homothety $h:\chain_3(2)\to \chain_3(4)$ indicated by the following figure.
	\begin{figure}[H]
		\centering
		\includegraphics{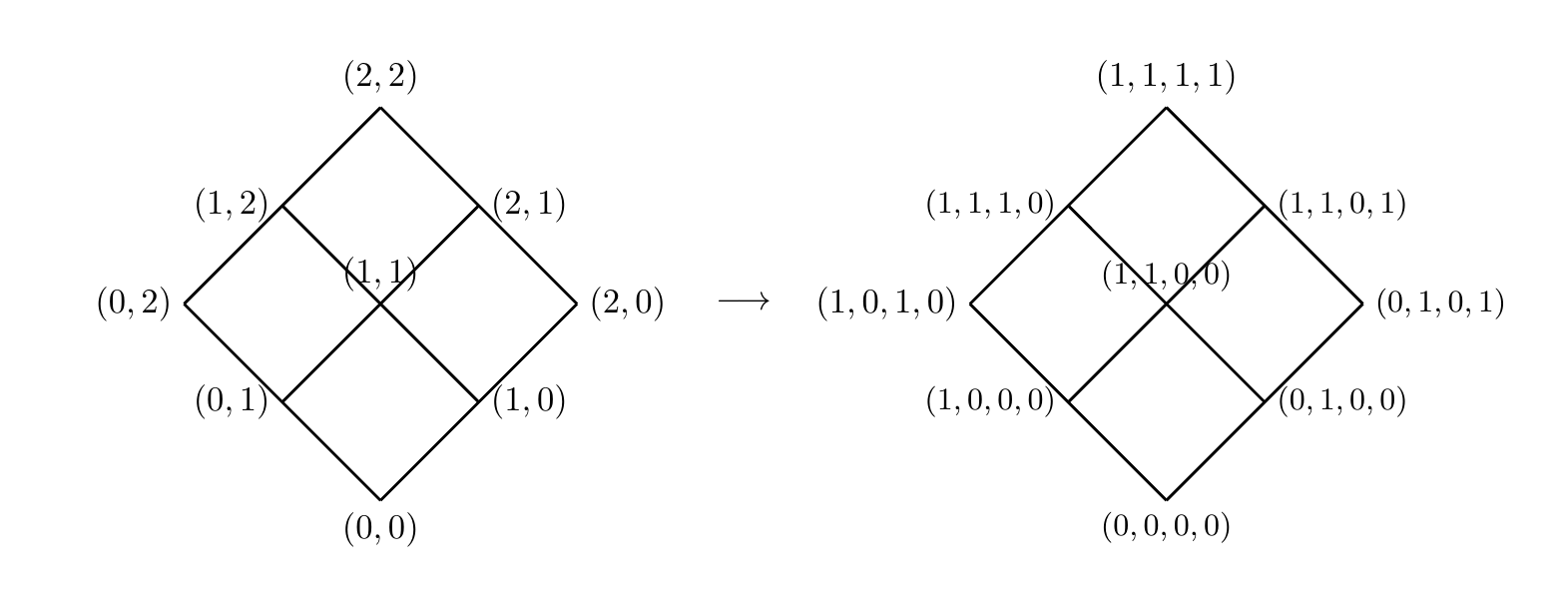}
		\caption{A homothety from $\chain_3(2)$ to $\chain_3(4)$.}
	\end{figure}
	\noindent
	It is easily checked that this is indeed a lattice homothety but is not of type HJ.
\end{example}

\begin{lemma}
	\label{lemma:chain latticec with hj system is good}
	The sequence of chain lattices equipped with the HJ homothety system is a good sequence.
\end{lemma}
\begin{proof}
	Let $t\geq 1$ be fixed.
	Let $F$ and $G$ be disjoint non-empty sets and consider the map
	$\textbf{i}:\chain_t(F)\times \chain_t(G)\to \chain_t(F\sqcup G)$ which takes $(f, g)$ in $\chain_t(F)\times \chain_t(G)$ to the map $F\sqcup G\to \set{0, \ldots, t-1}$ defined as
	\begin{equation}
		x\mapsto
		\left\{
		\begin{array}{lcl}
			f(x) \text{ if } x\in F\\
			g(x) \text{ if } x\in G
		\end{array}
		\right.
	\end{equation}
	For any function $\psi:G\to \set{0, \ldots, t-1}$, let use write $\tilde \psi$ to denote the map $F\sqcup G\to \set{0, \ldots, t-1}$ which vanishes on $F$ and takes $y\in G$ to $\psi(y)$.
	Similarly, for any map $\vp:F\to \set{0, \ldots, t-1}$ we define $\tilde \vp:F\sqcup G\to \set{0, \ldots, t-1}$.

	For $g_0\in \chain_t(G)$ fixed, we see that the map $\textbf{i}_{g_0}:\chain_t(F)\to \chain_t(F\sqcup G)$ takes $f\in \chain_t(F)$ to $\sum_{i\in F} f(i) 1_{\set{i}} + \tilde g_0$.
	Since $\tilde g_0$ vanishes on $F$, we see that $\textbf{i}_{g_0}$ is a homothety of type HJ.
	Similarly, $i_{f_0}:\chain_t(F)\to \set{0, \ldots, t-1}$ is a homothety of type HJ for each $f_0\in \chain_t(F)$.

	Finally, let $S$ be a finite set and $h_F:\chain_t(S)\to \chain_t(F)$ and $h_G:\chain_t(S)\to \chain_t(G)$ be homotheties of type HJ.
	Consider the map $h:\chain_t(S)\to \chain_t(F)\times \chain_t(G)$ defined as $h(s) = (h_F(s), h_G(s))$.
	
	We will show that $\textbf{i}\circ h$ is a homothety of type HJ.
	Let $|S|= n$.
	Since $h_F$ is a homothety of type HJ, there exists a family $\mc F=\set{F_i:\ i\in S}$ of pairwise disjoint subsets of $F$ of equal size, and a function $\alpha:F\to \set{0, \ldots, t-1}$ vanishing on $\bigcup \mc F$ such that $h_F=\vp_{\mc F, \alpha}$.
	Similarly, there is a family $\mc G=\set{G_i:\ i\in S}$  of pairwise disjoint subsets of $G$ of equal size, and a function $\beta:G\to \set{0, \ldots, t-1}$ vanishing of $\bigcup \mc G$ such that $h_G = \vp_{\mc G, \beta}$.
	It is readily checked that
	\begin{equation}
		\textbf{i}\circ h(s) = \sum_{i\in S} s(i) 1_{F_i\sqcup G_i} + \tilde \alpha + \tilde \beta
	\end{equation}
	Let $\mc E = \set{F_i\cup G_i:\ i\in S}$.
	It is clear that $\gamma:=\tilde \alpha+\tilde \beta$ vanihses on $\bigcup \mc E$ and that all elements of $\mc E$ are of equal size.
	Thus $\textbf{i}\circ h = \vp_{\mc E, \gamma}$ and we are done.
\end{proof}

\begin{corollary}
	\textit{\textbf{Hales-Jewett Theorem.}}
	Let $n, t$ and $k$ be positive integers.
	Then there is $N$ such that whenever $\chain_t(N)$ is colored with $k$ colors, there is a lattice homothety $h:\chain_t(n)\to \chain_t(N)$ of type HJ whose image is monochromatic.
\end{corollary}
\begin{proof}
	This follows immediately from Lemma \ref{lemma:chain latticec with hj system is good} and Theorem \ref{theorem:main}.
\end{proof}

\bibliographystyle{alpha}
\bibliography{RA.bib}
\end{document}